\documentclass [11pt, twoside]{amsart}
\usepackage[top=1.2in, bottom=1.2in, left=1.5in, right=1.5in]{geometry}
\usepackage{tikz}
\usepackage[active]{srcltx}
\usepackage[curve]{xy}
\usepackage{graphicx}
\usepackage{mathrsfs}
\usepackage{color}
\usepackage{xcolor}
\usepackage{amssymb,amsmath,amsthm,amsfonts}
\usepackage{enumerate}
\usepackage{mathtools}
\usepackage{physics}
\usepackage{mathrsfs}
\usepackage{thmtools}
\usepackage{hyperref}

\allowdisplaybreaks[4] 



\numberwithin{equation}{section}
\numberwithin{figure}{section}

\theoremstyle{plain}
\newtheorem{thm}{Theorem}[section]
\theoremstyle{plain}
\newtheorem*{thm*}{Theorem}

\theoremstyle{definition}

\theoremstyle{definition}
\newtheorem*{defn*}{Defnition}

\theoremstyle{remark}

\theoremstyle{remark}
\newtheorem*{rem*}{Remark}

\theoremstyle{remark}

\theoremstyle{plain}
\newtheorem{lem}[thm]{Lemma}

\theoremstyle{plain}
\newtheorem{prop}[thm]{Proposition}

\theoremstyle{plain}
\newtheorem{cor}[thm]{Corollary}

\theoremstyle{plain}

\begin{document}

\title[Dimension of Sets of Continued Fractions]{Hausdorff Dimension of Sets of Continued Fractions with Unbounded Partial Quotients Along Subsequence}

\author{Yuefeng Tang}


\begin{abstract}
	Let $x=[a_1(x),a_2(x),\ldots]$ be the continued fraction expansion of $x\in[0,1)$. We prove that the Hausdorff dimension of
	\begin{equation*}
		E_{even}=\{x\in[0,1)\colon a_{2n}(x)\to\infty\ (n\to\infty)\}.
	\end{equation*}
	is 1/2. In general, we study the set of continued fractions with unbounded partial quotients along subsequence
	\begin{equation*}
		E_{\{k_n\}}=\{x\in[0,1)\colon a_{k_n}(x)\to\infty\ (n\to\infty)\},
	\end{equation*}
	where $\{k_n\}\subset\mathbb{N}$ is a subsequence. We show that $E_{\{k_n\}}$ has Hausdorff dimension 1/2 or 1 according to whether the set of indices $\{k_n\}_{n\geq 1}$ has positive or zero upper density respectively.
\end{abstract}

\maketitle


\section{Introduction}\label{1}

It is known that every irrational number $x\in[0,1)$ admits a unique infinite continued fraction expansion
\begin{equation*}
	x=\cfrac{1}{a_1(x)+\cfrac{1}{a_2(x)+\cdots}},
\end{equation*}
where the unique sequence of integers $a_1(x),a_2(x),\ldots$ are called the \textit{partial quotients} of $x$. We write $x=[a_1(x),a_2(x),\ldots]$ for simplicity. Rational numbers $x\in[0,1)$ also have continued fraction expansions, taking the form
\begin{equation*}
	x=\cfrac{1}{a_1(x)+\cfrac{1}{a_2(x)+\cdots+\cfrac{1}{a_n(x)}}},
\end{equation*}
however the sequence of partial quotients $\{a_n(x)\}$ is finite. Except that we impose the additional constraint $a_n(x)>1$, the expansion is unique. Truncating the sequence of continued fractions at its $k$th term, we obtain an irreducible rational number
\begin{equation*}
	\frac{p_k(x)}{q_k(x)}=[a_1(x),a_2(x),\ldots,a_k(x)],
\end{equation*}
called the $k$th \textit{convergent} of $x$.

We focus on the set 
\begin{equation*}
	E_{even}\coloneqq\{x\in I\colon a_{2n}(x)\to \infty\ (n\to\infty)\},
\end{equation*}
where $I=[0,1)$. It is a typical fractal set of continued fractions with unbounded partial quotients. In fact, as early in 1941, Good \cite{MR0004878} studied the set
\begin{equation*}
	E\coloneqq\{x\in I\colon a_n(x)\to\infty\ (n\to\infty)\},
\end{equation*}
and proved that $\dim_H(E)=1/2$, where $\dim_H$ denotes the Hausdorff dimension. Note that $E$ is a subset of $E_{even}$, which gives an intuitive lower bound of Hausdorff dimension. 

The study of Hausdorff dimension on the sets in continued fractions is usually related to the following sets:
\begin{align*}
	E(M)&\coloneqq\{x\in I\colon a_{n}(x)\leq M,\ \forall n\in\mathbb{N}\},	\\
	F(M)&\coloneqq\{x\in I\colon a_{n}(x)\geq M,\ \forall n\in\mathbb{N}\}.
\end{align*}
The earliest paper on the Hausdorff dimension of such sets was Jarn\'{i}k's paradigmatic \cite{Jarnik1928-1929}, in which he established that for every $M\geq 8$
\begin{equation*}
	1-\frac{4}{M\log(2)}\leq\dim_H(E(M))\leq1-\frac{1}{8M\log(M)}.
\end{equation*}
Kurzweil improved the former bounds in his doctoral work \cite{MR51273} by proving that
\begin{equation*}
	1-\frac{0.99}{M}\leq\dim_H(E(M))\leq1-\frac{0.25}{M}
\end{equation*}
for $M\geq 1000$. Hensley \cite{MR2351741} further improved Kurzweil's result using functional analytic techniques and proved that
\begin{equation*}
	\dim_H(E(M))=1-\frac{6}{\pi^2}\frac{1}{M}-\frac{72}{\pi^4}\frac{\log(M)}{M^2}+O(\frac{1}{M^2}).
\end{equation*}
On the other hand, Good \cite{MR0004878} proved that for $M\geq 20$
\begin{equation*}
	\frac{1}{2}+\frac{1}{2\log(M+2)}\leq\dim_H(F(M))\leq\frac{1}{2}+\frac{\log\log(M-1)}{2\log(M-1)}.
\end{equation*}
Jaerisch and Kesseb\"{o}hmer \cite{MR2672614} gave an asymptotic improvement on Good's result by showing that as $M\to\infty$
\begin{equation*}
	\dim_H(F(M))\sim\frac{1}{2}+\frac{\log\log(M)}{2\log(M)}.
\end{equation*}
Recently, Das, Fishman, Simmons and Urba\'nski \cite{MR4540838} employed perturbations of the conformal iterated function system to approximate $\dim_H(E(M))$ and $\dim_H(F(M))$ with errors of order $O_p\left(\frac{\log^{p-1}(M)}{M^p}\right)$ and $O_p\left(\frac{\log\log(M)}{M^p\log(M)}\right)$, respectively.

Using a classical strategy established by Good, we approximate the upper bound of Hausdorff dimension of set
\begin{equation*}
	\{x\in I\colon a_{2n}(x)\geq M,\ \forall n\in\mathbb{N}\},
\end{equation*}
and obtain our first result.
\begin{thm}\label{main1}
	The Hausdorff dimension of the set $E_{even}$ is 1/2.
\end{thm}

In general, one might wonder what would happen if the partial quotients grew along an arbitrary subsequence of indices instead of $\{2n\}_{n\geq 1}$. Let
\begin{equation*}
	E_{\{k_n\}}\coloneqq\{x\in I\colon a_{k_n}(x)\to\infty\ (n\to\infty)\},
\end{equation*}
where $\{k_n\}_{n\geq 1}$ is a strictly increasing sequence of positive integer. Let 
\begin{equation*}
	k(n)\coloneqq\sharp\{k_j\}\cap\{1,2,\cdots,n\},
\end{equation*}
and denote the \textit{upper and lower density} of sequence $\{k_n\}$ by
\begin{equation*}
	\overline{\mathbf{d}}(\{k_n\})\coloneqq\limsup_{n\to\infty}\frac{k(n)}{n}\quad\text{and}\quad\underline{\mathbf{d}}(\{k_n\})\coloneqq\liminf_{n\to\infty}\frac{k(n)}{n}.
\end{equation*}
If the upper and lower density coincide, i.e., the limit exists, we call
\begin{equation*}
	\mathbf{d}(\{k_n\})\coloneqq\lim_{n\to\infty}\frac{k(n)}{n},
\end{equation*}
the \textit{density} of sequence $\{k_n\}$.
We completely determine the Hausdorff dimension of $E_{\{k_n\}}$ and describe the existence of a dimension jumping phenomenon.
\begin{thm}\label{main2}
	The Hausdorff dimension of the set $E_{\{k_n\}}$ is 1/2 or 1 according as the upper density of $\{k_n\}$ is positive or zero.
\end{thm}
If the upper and lower density are not equal, we believe that there might be some singular behavior on the set. For some relevant research or discussions, we refer to Fang, Ma and Song \cite{MR4341100} and Fang, Moreira and Zhang \cite{fang2024fractalgeometrycontinuedfractions}.

\subsection*{Structure of the paper}

The present paper is organized as follows. In Section \ref{2}, we introduce some notations, define necessary concepts and present some useful lemmas. In Section \ref{3}, we follow the classical way started by Good \cite{MR0004878} to determine the Hausdorff dimension of $E_{even}$. In Section \ref{4}, we separately prove the cases of zero and positive density in our second main Theorem \ref{main2}. 

\subsection*{Acknowledgement}

The author wishes to express appreciation to Professors Lingmin Liao and Yueli Yu for their invaluable comments and numerous suggestions on the paper. The author wishes to thank Cheng Liu for his suggestions for improving Theorem 1.2.

\section{Preliminary}\label{2}

Let $F$ be a subset of $\mathbb{R}^n$. For $\delta>0$, we call a countable (or finite) sequence of sets $\{U_i\}$ a \textit{$\delta$-cover} of set $F$, if $0<\abs{U_i}\leq\delta$ and $F\subset\cup_iU_i$.  Let $s\geq0$, we define
\begin{equation*}
	\mathcal{H}^s_\delta(F)\coloneqq\inf\Big\{\sum_{i=1}^{\infty}\abs{U_i}^s\colon\{U_i\}\ \text{is a $\delta$-cover of $F$}\Big\}.
\end{equation*}
Note that $\mathcal{H}^s_\delta(F)$ is a decreasing function with respect to $\delta$. We define the \textit{$s$-dimensional Hausdorff measure} of the set $F$ by
\begin{equation*}
	\mathcal{H}^s(F)=\lim_{\delta\rightarrow0}\mathcal{H}^s_{\delta}(F)=\sup_{\delta>0}\mathcal{H}^s_{\delta}(F).
\end{equation*}
For any set $F$, there exists a critical point of $s$ at which $\mathcal{H}^s(F)$ jumps from $\infty$ to $0$. Such critical point obtained by
\begin{equation*}
	\dim_H(F)=\inf\{s>0\colon\mathcal{H}^s(F)=0\}=\sup\{s>0\colon\mathcal{H}^s(F)=\infty\}
\end{equation*}
is called the \textit{Hausdorff dimension} of $F$.

Let $\frac{p_n}{q_n}=[a_0;a_1,a_2,\ldots,a_n]$ be the $n$th convergent of the real number $x=[a_0;a_1,a_2,a_3,\ldots]$. Then $p_k,q_k,0\leq k\leq n$, satisfy the following relations (see \cite{MR1451873})
\begin{align*}
	&p_{-1}=1;\quad	p_0=a_0;\quad	p_k=a_kp_{k-1}+p_{k-2},\quad	1\leq k\leq n,	\\
	&q_{-1}=0;\quad	q_0=1;\quad	q_k=a_kq_{k-1}+q_{k-2},\quad	1\leq k\leq n.	
\end{align*}
We have the identities (see \cite[Theorem 2]{MR1451873})
\begin{equation}
	\frac{p_{n-1}}{q_{n-1}}-\frac{p_n}{q_n}=\frac{(-1)^n}{q_nq_{n-1}},\quad n\geq1. \label{pre1}
\end{equation}
Given $a_1,a_2,\ldots,a_n,$ we call
\begin{equation*}
	I(a_1,a_2,\ldots,a_n)=\left\{\begin{aligned}
		&\left[\frac{p_n}{q_n},\frac{p_n+p_{n-1}}{q_n+q_{n-1}}\right),\quad\text{if $n$ is even},	\\
		&\left(\frac{p_n+p_{n-1}}{q_n+q_{n-1}},\frac{p_n}{q_n}\right],\quad\text{if $n$ is odd},	\\
	\end{aligned}\right.
\end{equation*}
a \textit{cylinder of order $n$}. In fact, $I(a_1,a_2,\ldots,a_n)$ just represents the set of numbers in $I$ which have a continued fraction expansion beginning with $a_1,a_2,\ldots,a_n$, i.e.,
\begin{equation*}
	I(a_1,a_2,\ldots,a_n)=\{x\in I\colon a_1(x)=a_1,a_2(x)=a_2,\ldots,a_n(x)=a_n\}.
\end{equation*}
It is well known ({see \cite[Theorem 12 and Theorem 13]{MR1451873}}) that 
\begin{equation}\label{pre2}
	q_n\geq2^{\frac{n-1}{2}}\quad\text{for}\quad n\geq1,\quad\text{and}\quad\abs{I(a_1,a_2,\ldots,a_n)}=\frac{1}{q_n(q_n+q_{n-1})}.
\end{equation}

In the following, we give some useful lemmas.
\begin{lem}[{\cite[Lemma 1]{MR0004878}}]\label{lemg1}
	Let $b_1,b_2,\ldots,b_n$ be $n$ given positive integers. For a subset $S\subset I$, we define $S^\prime$ to be the set of
	\begin{equation*}
		x^\prime=[b_1,b_2,\ldots,b_n,a_1,a_2,\ldots],
	\end{equation*}
	where $x=[a_1,a_2,\ldots]\in S$. Then $S^\prime\subset I(b_1,b_2,\ldots,b_n)$, and $\dim_H(S)=\dim_H(S^\prime)$.
\end{lem}

\begin{lem}[{\cite[Lemma 2]{MR0004878}}]\label{lemg2}
	Let $\{E_i\}_{i\in\mathbb{N}}$ be a sequence of nonempty sets of positive integers, and $T$ be the set of all values of $x=[a_1,a_2,\ldots]$, for which $a_i\in E_i$, $i\in\mathbb{N}$. Then, for any $n$ and any $b_1,\ldots,b_n$ with $b_i\in E_i$, $1\leq i \leq n$, $\dim_H(T\cap I(b_1,b_2,\ldots,b_n))=\dim_H(T)$.
\end{lem}
%

\begin{cor}[{\cite[Corollary]{MR0004878}}]\label{corg1}
	Let $T$ be a nonempty set defined as in Lemma \ref{lemg2}, and $U$ be a nonempty set defined in the same way except that a finite number of the sets $E_i$ are replaced by other sets $E_i^\prime$ of positive integers. Then, $\dim_H(U)=\dim_H(T)$.
\end{cor}

\begin{lem}[{\cite[Lemma 2.1]{MR2215567}}]\label{lem3}
	For any $n\geq1$ and $1\leq k\leq n$, we have
	\begin{equation*}
		\frac{a_k+1}{2}\leq \frac{q_n(a_1,a_2,\ldots,a_n)}{q_{n-1}(a_1,\ldots,a_{k-1},a_{k+1},\ldots,a_n)}\leq a_k+1.
	\end{equation*}
\end{lem}

\medskip

\section{Classical way to determine the Hausdorff dimension of $E_{even}$} \label{3}

Recall that
\begin{align*}
	E_{even}&=\{x\in I\colon a_{2n}(x)\to\infty\ (n\to\infty)\}	\\
	&=\bigcap_{M=1}^{\infty}\bigcup_{N=1}^{\infty}\bigcap_{n\geq N}\{x\in I\colon a_{2n}(x)\geq M\}
\end{align*}
To simplify our notation, we define the following sets 
\begin{align*}
	E_{even}(M)&=\{x\in I\colon a_{2n}(x)\geq M \text{, as }n\to\infty\},	\\
	E_{even}(M,N)&=\{x\in I\colon a_{2n}(x)\geq M,\ \forall n\geq N\}.
\end{align*}
Then, we have
\begin{equation*}
	E_{even}=\bigcap_{M=1}^{\infty}E_{even}(M)=\bigcap_{M=1}^{\infty}\bigcup_{N\geq1}^{\infty}E_{even}(M,N).
\end{equation*}
By Corollary \ref{corg1}, we have $\dim_H(E_{even}(M,N))=\dim_H(E_{even}(M,1))$ for any $N\in\mathbb{N}$. Therefore, we have $\dim_H(E_{even}(M))=\dim_H(E_{even}(M,1))$, which implies $\dim_H(E_{even})\leq\dim_H(E_{even}(M,1))$.

It is observed that $E_{even}(1,1)\supset E_{even}(2,1)\supset E_{even}(3,1)\supset\cdots$. Hence the sequence $\{\dim_H(E_{even}(M,1))\}_{M\in\mathbb{N}}$ is decreasing. A classical method developed by Good \cite{MR0004878} tells us that we only need to consider the limiting behavior of $\dim_H(E_{even}(M,1))$ as $M\to \infty$. Note that 
\begin{align*}
	E_{even}(M,1)&=\{x\in I\colon a_{2n}(x)\geq M,\ \forall n\in\mathbb{N}\}	\\
	&=\bigcap_{N=1}^{\infty}\{x\in I\colon a_{2n}(x)\geq M, n\leq N\}.
\end{align*}
Let $J(a_1,a_2,\ldots,a_{2n-1})=\cup_{a_{2n}\geq M}I(a_1,a_2,\ldots,a_{2n})$. Then
\begin{equation*}
	\mathcal{H}^s(E_{even}(M,1))\leq\liminf_{n\to\infty}\sum_{\substack{a_2,a_4,\ldots,a_{2n-2}\geq M \\ a_1,a_3,\ldots,a_{2n-1}\in\mathbb{N}}}\abs{J(a_1,a_2,\ldots,a_{2n-1})}^s.
\end{equation*}
To simplify our notation, we write $q_{n}=q_{n}(a_1,a_2,\ldots,a_{n})$ without causing ambiguity. Observe that
\begin{align*}
	\abs{J(a_1,a_2,\ldots,a_{2n-1})}&= \frac{p_{2n}(a_1,\ldots,a_{2n-1}, M)}{q_{2n}(a_1,\ldots,a_{2n-1},M)}-\frac{p_{2n-1}}{q_{2n-1}}	\\
	&=\frac{q_{2n-2}(a_1,\ldots,a_{2n-3},M)q_{2n-3}}{q_{2n}(a_1,\ldots,a_{2n-1},M)q_{2n-1}}\cdot\abs{J(a_1,a_2,\ldots,a_{2n-3})}	\\
	&\leq\frac{(Mq_{2n-3}+q_{2n-4})q_{2n-3}}{Ma_{2n-1}^2a_{2n-2}^2q_{2n-3}^2}\cdot\abs{J(a_1,a_2,\ldots,a_{2n-3})}	\\
	&\leq\frac{M+1}{Ma_{2n-1}^2a_{2n-2}^2}\cdot\abs{J(a_1,a_2,\ldots,a_{2n-3})}.
\end{align*}
Then, for any $s$, $M$, and $n$, we have
\begin{align*}
	&\sum_{\substack{a_2,a_4,\ldots,a_{2n-2}\geq M \\ a_1,a_3,\ldots,a_{2n-1}\in\mathbb{N}}}\abs{J(a_1,a_2,\ldots,a_{2n-1})}^s	\\
	\leq&\sum_{\substack{a_2,a_4,\ldots,a_{2n-4}\geq M \\ a_1,a_3,\ldots,a_{2n-3}\in\mathbb{N}}}\abs{J(a_1,a_2,\ldots,a_{2n-3})}^s\cdot\sum_{\substack{a_{2n-1}\in\mathbb{N} \\ a_{2n-2}\geq M}}\left(\frac{M+1}{Ma_{2n-1}^2a_{2n-2}^2}\right)^s.
\end{align*}
To analyze the Hausdorff dimension of $E_{even}$, we focus on the formula 
\begin{equation}\label{3.1}
	\sum_{\substack{a_{2n-1}\in\mathbb{N} \\ a_{2n-2}\geq M}}\left(\frac{M+1}{Ma_{2n-1}^2a_{2n-2}^2}\right)^s\leq 1, 
\end{equation}
which makes sure that $\mathcal{H}^s(E_{even}(M,1))$ converges. To find the critical value $s$ for which the sum converges, we analyze the asymptotic behavior of (\ref{3.1}) as $M\to\infty$. The critical equation is as follows
\begin{equation*}
	\sum_{i\geq M}\sum_{j\in\mathbb{N}}\left(\frac{M+1}{Mi^2j^2}\right)^s=\left(1+\frac{1}{M}\right)^s\zeta(2s)\sum_{k=M}^{\infty}\frac{1}{k^{2s}}=1,
\end{equation*}
where $\zeta(s)\coloneqq\sum_{k=1}^{\infty}k^{-s}$ is the well-known \textit{Riemann zeta function}. 

We now refine the asymptotic approximations. To simplify our calculations, we use the Landau notations $O(\cdot)$ and $o(\cdot)$. Note that
\begin{equation}\label{3.2}
	\begin{aligned}
		\left(1 + \frac{1}{M}\right)^s &= \exp\left(s \ln\left(1 + \frac{1}{M}\right)\right)	\\
		&= \exp\left(s \left( \frac{1}{M} - \frac{1}{2M^2} + O(M^{-3}) \right)\right)	\\
		&= e^{s/M} \left(1 + O_s(M^{-2})\right),
	\end{aligned}
\end{equation}
where the constant in the asymptotic notation $O_s(\cdot)$ depends on $s$.
For the tail of the Riemann zeta function, we have
\begin{equation}\label{3.3}
	\sum_{k=M}^{\infty} k^{-2s} = \int_M^{\infty} x^{-2s}  dx + O(M^{-2s}) = \frac{M^{1-2s}}{2s-1} + O(M^{-2s}),
\end{equation}
where $s>1/2$. Set $s = \frac{1}{2} + \delta$. Using the Laurent series, we have
\begin{equation}\label{3.4}
	\zeta(2s) = \zeta(1 + 2\delta) = \frac{1}{2\delta} + \gamma + O(\delta),
\end{equation}
where $\gamma$ is the \textit{Euler--Mascheroni constant} and the constant in $O(\delta)$ is absolute. Substituting the approximations (\ref{3.2})--(\ref{3.4}) into the critical equation, we have
\begin{equation*}
	e^{s/M} \left(1 + O_s(M^{-2})\right) \cdot \left( \frac{1}{2\delta} + \gamma + O(\delta) \right) \cdot \left( \frac{M^{1-2s}}{2s-1} + O(M^{-2s}) \right) = 1,
\end{equation*}
which is equivalent to
\begin{equation}\label{3.5}
	e^{{s}/{M}} \cdot \frac{M^{-2\delta}}{(2\delta)^2} \cdot \left(1 + 2\gamma\delta + O(\delta^2) + O_{\delta}(M^{-1})\right) = 1.
\end{equation}
Note that the term $O_{\delta}(M^{-1})$ is $o(1)$ as $M\to\infty$ for any $\delta>0$. To find the critical $\delta$ as a function of $M$, taking the natural logarithms of both sides of (\ref{3.5}) gives
\begin{equation*}
	\frac{1}{2M} + \frac{\delta}{M} - 2\delta \log M - 2\log(2\delta) + \log(1+2\gamma\delta +O(\delta^2)+o(1)) = 0.
\end{equation*}
The dominant balance is between $-2\delta \log M$ and $-2\log(2\delta)$ as $M$ tends to infinity, i.e., 
\begin{equation}\label{3.6}
	-2\delta \log M \sim 2\log(2\delta),
\end{equation}
where the notation $f\sim g$ represents $\lim_{M\to\infty}f(M)/g(M)=1$ . Solving (\ref{3.6}) asymptotically, we get
\begin{equation*}
	\delta \sim -\frac{\log(2\delta)}{\log M} \sim \frac{\log \log M - \log2}{\log M} + O\left(\frac{\log \log \log M}{\log M}\right).
\end{equation*}
Thus, the critical exponent satisfies
\begin{equation*}
	s = \frac{1}{2} + \delta \sim \frac{1}{2} + \frac{\log \log M}{\log M} - \frac{\log 2}{\log M} + O\left(\frac{\log \log \log M}{\log M}\right),
\end{equation*}
which is an asymptotic approximation of the upper bound of the Hausdorff dimension of $E_{even}(M,1)$. As $M \to \infty$, we finish the proof of our first main Theorem \ref{main1}.

\section{Proof of Theorem \ref{main2} and extension to Hirst-type sets}\label{4}

In this section, we study the Hausdorff dimension of the set 
\begin{equation*}
	E_{\{k_n\}}=\{x\in I\colon a_{k_n}(x)\to\infty\ (n\to\infty) \},
\end{equation*}
where $\{k_n\}_{k\geq 1}$ is a strictly increasing sequence of positive integers. Some classical results show that the Hausdorff dimension of $E_{\{k_n\}}$ is closely related to the density of $\{k_n\}_{n\geq 1}$. We can infer from Wu and Xu \cite{MR2776116} that $\dim_H(E_{\{n^2\}})=1$, which is a specific example of our second main theorem on the case of zero density.

\subsection{The case of zero density}

We first consider the case of zero density, since the zero upper density implies the zero density. Our method is motivated by Wu and Xu \cite{MR2776116}, and we generalize Wu and Xu's approach using an updated construction.

Let
\begin{equation*}
	E_{\{k_n\}}(M)\coloneqq\{x\in I\colon 1\leq a_j(x)\leq M,\ j\notin\{k_n\},\ a_{k_n}(x)\to\infty\ (n\to\infty)\},
\end{equation*}
which is a subset of $E_{\{k_n\}}$. Let $\phi\colon\mathbb{N}\to D\subseteq\mathbb{N}$ be an increasing function to be chosen later and define the submersion $f\colon\mathcal{D}\to\overline{\mathcal{D}}$ by
\begin{equation*}
	(\sigma_1,\ldots,\sigma_n)\mapsto\overline{(\sigma_1,\ldots,\sigma_n)}\coloneqq(\sigma_1,\ldots,\widehat{\sigma_{k_j}},\ldots,\sigma_n),
\end{equation*}
where $\widehat{\sigma_{k_j}}$ means eliminating the terms $\{\sigma_{k_j}\colon 1\leq k_j\leq n\}$ in $(\sigma_1,\ldots,\sigma_n)$. For any $n\geq 1$, set
\begin{align*}
	D_n&\coloneqq\{(\sigma_1,\ldots,\sigma_n)\in\mathbb{N}^{n}\colon 1\leq\sigma_i\leq M \text{ and } \sigma_{k_j}=\phi(j) \text{ for all } 1\leq k_j\leq n\},	\\
	\mathcal{D}&\coloneqq\bigcup_{n=0}^{\infty}D_n,\quad (D_0\coloneqq\varnothing),	\\
	\overline{D_n}&\coloneqq\{(\sigma_1,\ldots,\sigma_n)\in\mathbb{N}^{n}\colon 1\leq\sigma_i\leq M\},	\\
	\overline{\mathcal{D}}&\coloneqq\bigcup_{n=0}^{\infty}\overline{D_n},\quad (\overline{D_0}\coloneqq\varnothing).
\end{align*}
We define the submersion of a cylinder of order $n$ of continued fractions by
\begin{equation*}
	\overline{I_n}(\sigma_1,\ldots,\sigma_n)\coloneqq I_{n-k(n)}\overline{(\sigma_1,\ldots,\sigma_n)}.
\end{equation*}

To obtain the desired result, it suffices to construct subsets of dimension arbitrarily close to 1. Let
\begin{equation*}
	E_{\{k_n\},\phi}(M)\coloneqq\left\{x\in I\colon
	1\leq a_j(x)\leq M,\ j\notin\{k_n\},\ a_{k_n}(x)=\phi(n),\ n\in\mathbb{N} \right\}.
\end{equation*}
The values of $\phi(n)$ are chosen by the following procedure.
\begin{lem}\label{njchoice}
	Let $\{k_n\}_{n\geq 1}$ be a strictly increasing sequence of positive integers with zero density. Then, for any given constant $c_1>0$, we have a sequence $\{n_j\}_{j\geq 1}$ of indices such that 
	\begin{equation*}
		c_1 k(n)\geq \sum_{j=1}^{k(n)}\log(\phi(j)+1)
	\end{equation*}
	holds for all $n$, where $k(n)=\sharp\{k_j\}\cap\{1,2,\cdots,n\}$ and $\phi(n)$ is defined as follows:
	\begin{equation}\label{4.1}
		\begin{aligned}
			\phi(1)=\cdots=\phi(n_1)=1,	\\
			\phi(n_1+1)=\cdots=\phi(n_2)=2,	\\
			\vdots	\\
			\phi(n_{j-1}+1)=\cdots=\phi(n_j)=j,	\\
			\vdots
		\end{aligned}
	\end{equation}
\end{lem}
\begin{proof}
	Since $\lim_{n\to\infty}\frac{k(n)}{n}=0$, there exists a sequence of positive integers $\{N_j\}_{j\geq 1}$ such that for all $n>N_j$
	\begin{equation*}
		\frac{k(n)}{n}\leq\frac{c_1}{\log(j+1)}.
	\end{equation*}
	Let $n_j$ be the least integer such that $k_{n_j}\geq N_j$ and $n_j>n_{j-1}$. Then, we can determine $\{n_j\}$ by induction.
	
	For a fixed $n$, there exists a unique $j$ satisfying $n_{j-1}<k(n)\leq n_j$. Thus
	\begin{align*}
		\sum_{j=1}^{k(n)}\log(\phi(j)+1)&= n_1\log2+(n_2-n_1)\log3+\ldots+(k(n)-n_{j-1})\log(j+1)	\\
		&\leq k(n)\log(j+1)	\\
		&\leq c_1n.
	\end{align*}
	Since $n$ is arbitrary, we finish the proof.
\end{proof}

\begin{lem}\label{1+epsilon}
	For any $\varepsilon>0$, there exists $N_0=N_0(\varepsilon)$ such that for any $n\geq N_0$ and any $(\sigma_1,\ldots,\sigma_n)\in{D}_n$, we have 
	\begin{equation*}
		\abs{I_n(\sigma_1,\ldots,\sigma_n)}\geq\abs{\overline{I_n}(\sigma_1,\ldots,\sigma_n)}^{1+\varepsilon}.
	\end{equation*}
\end{lem}
\begin{proof}
	For any $\varepsilon>0$, by Lemma \ref{njchoice}, setting $c_1=\varepsilon\log(2)/2$, we can determine the series of indices $\{n_j\}$ and $\phi(n)$ immediately. Since $\lim_{n\to\infty}\frac{k(n)}{n}=0$, there exists an $N_0=N_0(\varepsilon)$ such that for all $n\geq N_0$,
	\begin{equation*}
		\frac{1}{2}n\cdot\varepsilon\log2\geq\log2+2\varepsilon\log2+k(n)\cdot\varepsilon\log2.
	\end{equation*}
	Then, we have 
	\begin{equation*}
		q_{n-k(n)}^{2\varepsilon}\overline{(\sigma_1,\ldots,\sigma_n)}\geq 2^{(n-k(n)-2)\varepsilon}\geq 2\prod_{j=1}^{k(n)}(\sigma_{k_j}+1)^2,
	\end{equation*}
	for any $(\sigma_1,\ldots,\sigma_n)\in\mathcal{D}_n$.
	Thus, by Lemma \ref{lem3}, we have
	\begin{align*}
		\abs{I_n(\sigma_1,\ldots,\sigma_n)}&\geq\frac{1}{2q_n^2(\sigma_1,\ldots,\sigma_n)}	\\
		&\geq\frac{1}{2q_{n-k(n)}^2\overline{(\sigma_1,\ldots,\sigma_n)}\prod_{j=1}^{k(n)}(\sigma_{k_j}+1)^2}	\\
		&\geq\frac{1}{q_{n-k(n)}^{2+2\varepsilon}\overline{(\sigma_1,\ldots,\sigma_n)}}	\\
		&\geq\abs{\overline{I_n}(\sigma_1,\ldots,\sigma_n)}^{1+\varepsilon}.
	\end{align*}
\end{proof}

\begin{lem}\label{y-xdistance}
	Given $n\geq N_0$, and $(\sigma_1,\ldots,\sigma_n)\in\mathcal{D}_n$, for any $x=[\sigma_1,\ldots,\sigma_n,x_{n+1},\ldots]\in E_{\{k_n\},\phi}(M)$ and any $y=[\sigma_1,\ldots,\sigma_n,y_{n+1},\ldots]\in E_{\{k_n\},\phi}(M)$ with $x_{n+1}\ne y_{n+1}$, we have
	\begin{equation*}
		\abs{x-y}\geq \frac{1}{9M^3}\abs{I_n(\sigma_1,\ldots,\sigma_n)}.
	\end{equation*}
\end{lem}
\begin{proof}
	Without loss of generality, we assume that $x<y$ and $n$ is even. We distinguish two cases.
	
	\textbf{Case I.} $n+2=k_j$ for some $j\in\mathbb{N}$.
	
	In this case, the distance between $x$ and $y$ is greater than the distance between the right end point of $I(\sigma_1,\ldots,\sigma_n,x_{n+1},\phi(j))$ and the left end point of $I(\sigma_1,\ldots,\sigma_n,y_{n+1},\phi(j))$. Thus,
	\begin{align*}
		y-x&\geq\abs{\frac{(y_{n+1}+\frac{1}{\phi(j)})p_n+p_{n-1}}{(y_{n+1}+\frac{1}{\phi(j)})q_n+q_{n-1}}-\frac{(x_{n+1}+\frac{1}{\phi(j)+1})p_n+p_{n-1}}{(y_{n+1}+\frac{1}{\phi(j)+1})q_n+q_{n-1}}}	\\
		&=\frac{\abs{y_{n+1}-x_{n+1}+\frac{1}{\phi(j)}-\frac{1}{\phi(j)+1}}}{\left((y_{n+1}+\frac{1}{\phi(j)})q_n+q_{n-1}\right)\left((x_{n+1}+\frac{1}{\phi(j)+1})q_n+q_{n-1}\right)}	\\
		&\geq\frac{1}{9M^2q_n^2}	\\
		&\geq\frac{1}{9M^2}\abs{I_n(\sigma_1,\ldots,\sigma_n)}.
	\end{align*}
	
	\textbf{Case II.} $n+2\notin\{k_j\}$.
	
	Similarly, we have
	\begin{align*}
		y-x&\geq\abs{\frac{(y_{n+1}+\frac{1}{y_{n+2}})p_n+p_{n-1}}{(y_{n+1}+\frac{1}{y_{n+2}})q_n+q_{n-1}}-\frac{(x_{n+1}+\frac{1}{x_{n+2}+1})p_n+p_{n-1}}{(y_{n+1}+\frac{1}{x_{n+2}+1})q_n+q_{n-1}}}	\\
		&=\frac{\abs{y_{n+1}-x_{n+1}+\frac{1}{y_{n+2}}-\frac{1}{x_{n+2}+1}}}{\left((y_{n+1}+\frac{1}{y_{n+2}})q_n+q_{n-1}\right)\left((x_{n+1}+\frac{1}{x_{n+2}+1})q_n+q_{n-1}\right)}	\\
		&\geq\frac{1}{9M^3q_n^2}	\\
		&\geq\frac{1}{9M^3}\abs{I_n(\sigma_1,\ldots,\sigma_n)}.
	\end{align*}
\end{proof}

In the following, we prove the case of zero density in Theorem \ref{main2}. 

	The submersion $f\colon \mathcal{D}\to\overline{\mathcal{D}}$ induces a map $\widetilde{f}\colon E_{\{k_n\},\phi}(M)\to E(M)$. For any $x=[\sigma_1,\sigma_2,\ldots,\sigma_n,\ldots]\in E_{\{k_n\},\phi}$, let $\widetilde{f}(x)=\tilde{x}\coloneqq\lim_{n\to\infty}\overline{[\sigma_1,\sigma_2,\ldots,\sigma_n]}$. By Lemmas \ref{1+epsilon} and \ref{y-xdistance}, we know that when $x,y\in E_{\{k_n\},\phi}(M)$ satisfy the condition in Lemma \ref{y-xdistance}, \begin{equation*}
		\abs{x-y}\geq\frac{1}{9M^3}\abs{I_n(\sigma_1,\ldots,\sigma_n)}\geq\frac{1}{9M^3}\abs{\overline{I_n}(\sigma_1,\ldots,\sigma_n)}^{1+\varepsilon}.
	\end{equation*}
	By definition, we immediately know that $\widetilde{f}(x),\widetilde{f}(y)\in\overline{I_n}(\sigma_1,\ldots,\sigma_n)$ and 
	\begin{equation*}
		\abs{\widetilde{f}(x)-\widetilde{f}(y)}\leq\abs{\overline{I_n}(\sigma_1,\ldots,\sigma_n)}.
	\end{equation*}
	Thus, \begin{equation*}
		\abs{\widetilde{f}(x)-\widetilde{f}(y)}\leq(9M^3)^{\frac{1}{1+\varepsilon}}\cdot\abs{x-y}^{\frac{1}{1+\varepsilon}}.
	\end{equation*}
	Observing $\widetilde{f}(E_{\{k_n\},\phi}(M))=E(M)$, by \cite[Proposition 2.3]{MR3236784}, we have
	\begin{equation*}
		\dim_H(E_{\{k_n\}}(M))\geq\frac{1}{1+\varepsilon}\dim_H(E(M)).
	\end{equation*}
	Since $\varepsilon$ is arbitrary, together with the fact that $\dim_H(E(M))\to 1$ as $M\to\infty$, we conclude that $\dim_H(E_{\{k_n\}})=1$.
	

\subsection{The case of positive upper density}

Indeed, we can partially generalize our result to Hirst's problem \cite{MR311581} in the setting of $\xi$-regular infinite iterated function system and we refer to \cite{MR4880588} and reference therein for some recent progress. For an infinite subset $D\subseteq\mathbb{N}$, Hirst considered the set of continued fractions
\begin{equation*}
	E(D)\coloneqq\{x\in I\colon a_n(x)\in D,\ \forall n\in\mathbb{N},\ \text{and}\ a_n(x)\to\infty\ (n\to\infty)\}.
\end{equation*}
Let $\tau(D)$ be the \textit{exponent of convergence} of the series $\{{1}/{n}\}_{n\geq1}$, i.e.,
\begin{equation*}
	\tau(D)\coloneqq\inf\left\{s\geq 0\colon \sum_{n\in D}\frac{1}{n^s}<\infty\right\}.
\end{equation*}
Hirst proved that $\dim_H(E(D))\leq\tau(D)/2$, and conjectured that $\tau(D)/2$ is the exact value of the Hausdorff dimension. The conjecture was resolved by Wang and Wu \cite{MR2409173}. Our generalization is based on their result.

\begin{prop}\label{main3}
	Let $D$ be an infinite subset of $\mathbb{N}$ and $\{k_n\}_{n\geq 1}$ be a set of indices with positive upper density. The Hausdorff dimension of set 
	\begin{equation*}
		E_{\{k_n\}}(D)\coloneqq\{x\in I\colon a_n(x)\in D,\ \forall n\in\mathbb{N},\ \textnormal{and}\ a_{k_n}(x)\to\infty\ (n\to\infty) \}
	\end{equation*}
	is equal to $\tau(D)/2$.
\end{prop}
\begin{proof}[Proof of Proposition \ref{main3}]
	Note that ${E}(D)\subset{E}_{\{k_n\}}(D)$, which gives a lower bound of Hausdorff dimension of ${E}_{\{k_n\}}(D)$. For the upper bound, we directly consider the canonical covering system:
	\begin{align*}
		{E}_{\{k_n\}}(D)&\subset\bigcap_{M=1}^{\infty}\bigcup_{N=1}^{\infty}\{x\in I\colon a_n(x)\in D,\ \forall n\in\mathbb{N},\ \text{and}\ a_{k_n}(x)\geq M,\ \forall n\geq N\},	\\
		&=\bigcap_{M=1}^{\infty}\bigcup_{N=1}^{\infty}\bigcup_{a_1,\ldots,a_{k_N}}{E}(M,N,a_1,\ldots,a_{k_N}),
	\end{align*}
	where
	\begin{equation*}
		{E}(M,N,a_1,\ldots,a_{k_N})\coloneqq\left\{\begin{aligned}
			&a_i(x)=a_i,\ 1\leq i\leq k_N, \\
			x\in I\colon &a_n(x)\in D,\ \forall n> k_N,	\\
			&a_{k_j}(x)\geq M,\ \forall j\geq N
		\end{aligned} \right\}.
	\end{equation*}
	Recall that the upper density of sequence $\{k_n\}$ is
	\begin{equation*}
		\overline{\mathbf{d}}(\{k_n\})=\limsup_{n\to\infty}\frac{\sharp\{k_j\}\cap\{1,2,\cdots,n\}}{n}.
	\end{equation*}
	By the definition of $\tau(D)$, for any $0<\varepsilon<\overline{\mathbf{d}}(\{k_n\})$, we have 
	\begin{equation*}
		\sum_{a_i\in D}\frac{1}{{a_i}^{\tau(D)(1+\varepsilon)}}<\infty.
	\end{equation*}
	We can choose $n_0=n_0(\varepsilon)$, such that there exist infinitely many $n\geq n_0$ satisfying
	\begin{equation*}
		\frac{k(n)}{n}>\overline{\mathbf{d}}(\{k_n\})-\varepsilon>0.
	\end{equation*}
	Then, we choose $M_0=M_0(\varepsilon)$, such that for all $M\geq M_0$
	\begin{equation*}
		\left(\sum_{a_i\in D}\frac{1}{{a_i}^{\tau(D)(1+\varepsilon)}}\right)^{\frac{1}{\overline{\mathbf{d}}(\{k_n\})-\varepsilon}-1}\cdot\sum_{a_{k_j}\geq M,\ a_{k_j}\in D}\frac{1}{{a_{k_j}}^{\tau(D)(1+\varepsilon)}}\leq 1,
	\end{equation*}
	since $\sum_{a_{k_j}\geq M,\ a_{k_j}\in D}\frac{1}{{a_{k_j}}^{\tau(D)(1+\varepsilon)}}$ tends to zero as $M$ increases. We directly estimate the $s$-dimensional Hausdorff measure
	\begin{align*}
		&\mathcal{H}^{s}({E}(M,N,a_1,\ldots,a_{k_N}))	\\
		\leq&\liminf_{n\to\infty}\sum_{\substack{a_i\in D,\ k_N< i\leq k_n \\ a_{k_j}\geq M,\ N< j\leq n}}\abs{I(a_1,a_2,\ldots,a_{k_n})}^{s}	\\
		\leq&\liminf_{n\to\infty}\prod_{i=1}^{k_N}\frac{1}{a_i^{2s}}\cdot\sum_{\substack{a_i\in D,\ k_N< i\leq k_n \\ a_{k_j}\geq M,\ N< j\leq n}}\prod_{i=k_N+1}^{k_n}\frac{1}{a_i^{2s}}	\\
		=&\liminf_{n\to\infty}\prod_{i=1}^{k_N}\frac{1}{a_i^{2s}}\cdot\left(\sum_{a_i\in D}\frac{1}{a_i^{2s}}\right)^{(k_n-k_N)-(n-N)}\cdot\left(\sum_{a_{k_j}\geq M,\ a_{k_j}\in D}\frac{1}{a_{k_j}^{2s}}\right)^{n-N}	\\
		\leq&\liminf_{n\to\infty}\prod_{i=1}^{k_N}\frac{1}{a_i^{2s}}\left(\left(\sum_{a_i\in D}\frac{1}{a_i^{\tau(D)(1+\varepsilon)}}\right)^{\frac{1}{\overline{\mathbf{d}}(\{k_n\})-\varepsilon}-1}\cdot\sum_{a_{k_j}\geq M,\ a_{k_j}\in D}\frac{1}{a_{k_j}^{\tau(D)(1+\varepsilon)}}\right)^n<\infty,
	\end{align*}
	where $s=\tau(D)(1+\varepsilon)/2$. Since $\varepsilon$ is arbitrary, we conclude that $\dim_H({E}_{\{k_n\}}(D))=\tau(D)/2$.
\end{proof}

By Proposition \ref{main3}, we immediately prove the case of positive upper density in Theorem \ref{main2}.

\bibliographystyle{plain}
\bibliography{bibtex}

$\newline$\textsc{Yuefeng Tang, School of Mathematics and Statistics, Wuhan University, Wuhan, Hubei, China}
$\newline$\textit{E-mail: }\texttt{tangyuefeng2001@whu.edu.cn}

\end{document}